\documentclass[12pt,reqno]{amsart}
\usepackage[margin=1in]{geometry}
\usepackage{amscd,amsfonts,amsmath,amssymb,amsthm,latexsym,mathrsfs,textcomp,verbatim}
\usepackage{accents}
\usepackage{bm}
\usepackage{bbm}
\usepackage{booktabs}
\usepackage{cite}
\usepackage{color}
\usepackage{constants}
\usepackage{csquotes}
\usepackage{enumerate}
\usepackage{etoolbox}
\usepackage{float}
\usepackage{hypbmsec}
\usepackage[dvipsnames]{xcolor}
\usepackage[breaklinks=true,colorlinks=true,linkcolor=black!25!RoyalBlue,citecolor=black!25!RoyalBlue,filecolor=RoyalBlue,urlcolor=Violet]{hyperref}
\hypersetup{linktocpage}
\usepackage[capitalise]{cleveref}
\usepackage{mathtools}
\usepackage{soul}
\usepackage{tikz}
\usepackage{tikz-cd}
\usetikzlibrary{shapes.geometric}
\usetikzlibrary{shapes.misc}
\usetikzlibrary{positioning}
\allowdisplaybreaks[4]
\usepackage{hypbmsec}
\newtheorem{theorem}{Theorem}[section]
\newtheorem{corollary}[theorem]{Corollary}
\newtheorem{proposition}[theorem]{Proposition}
\newtheorem{lemma}[theorem]{Lemma}
\newtheorem{lettertheorem}{Theorem}

\newtheorem{lettercorollary}{Corollary}

\newtheorem{conjecture}[theorem]{Conjecture}

\newtheorem{remark}[theorem]{Remark}
\newtheorem{definition}[theorem]{Definition}
\newtheorem{question*}{Question}
\newtheorem{problem*}{Problem}

\theoremstyle{definition}

\numberwithin{equation}{section}

\newcommand{\GL}{\mathrm{GL}}

\newcommand{\rk}{\mathrm{rk}}

\newcommand{\Li}{\mathrm{Li}}
\renewcommand{\Im}{\mathrm{Im}}

\renewcommand{\pmod}[1]{\, (\mathrm{mod} {\, #1})}

\newcommand{\ord}{\mathop{\mathrm{ord}}}

\renewcommand{\Re}{\mathrm{Re}}

\def\res{\mathop{\mathrm{res}}}

\patchcmd{\section}{\scshape}{\bfseries}{}{}
\makeatletter
\renewcommand{\@secnumfont}{\bfseries}
\makeatother

\makeatletter\newcommand{\tpmod}[1]{{\@displayfalse\pmod{#1}}}
\makeatletter
\@namedef{subjclassname@2020}{\textup{2020} Mathematics Subject Classification}
\makeatother

\begin{document}

\title{Euler Product Asymptotics for $L$-functions of elliptic curves}

\author{Arshay Sheth}

\address[Arshay Sheth]{Mathematics Institute, Zeeman Building, University of Warwick, Coventry, CV4 7AL, UK}
\curraddr{School of Mathematics, Tata Institute of Fundamental Research, Homi Bhabha Road,
Mumbai - 400005, India}

\thanks{The author was supported by funding from the European Research Council under the European Union’s Horizon 2020 research and innovation programme (Grant agreement No. 101001051 — Shimura varieties and the Birch--Swinnerton-Dyer conjecture) during the writing of this paper.}
\email{arshay.sheth@warwick.ac.uk, asheth@math.tifr.res.in }
\urladdr{\href{https://sites.google.com/view/arshaysheth/home}{https://sites.google.com/view/arshaysheth/home}}

\begin{abstract}
Let $E/\mathbb Q$ be an elliptic curve and for each prime $p$, let $N_p$ denote the number of points of $E$ modulo $p$. The original version of the Birch and Swinnerton-Dyer conjecture asserts that $\prod \limits _{p \leq x} \frac{N_p}{p} \sim C (\log x) ^{\text{rank}(E(\mathbb Q))}$ as $x \to \infty$. Goldfeld (1982) showed that this conjecture implies both the Riemann Hypothesis for $L(E, s)$ and the modern formulation of the conjecture i.e. that $\ord \limits _{s=1} L(E, s)= \text{rank}(E(\mathbb Q))$. In this paper, we prove that if we let $r=\ord \limits_{s=1}L(E, s)$, then under the assumption of the Riemann Hypothesis for $L(E, s)$, 
we have that $\prod \limits _{p \leq x} \frac{N_p}{p}  \sim C (\log x)^r
$
for all $x$ outside a set of finite logarithmic measure. As corollaries, we recover not only Goldfeld's result, but we also prove a result in the direction of the converse. Our method of proof is based on establishing the asymptotic behaviour of partial Euler products of $L(E, s)$ in the right-half of the critical strip.  
\end{abstract}

\maketitle
%\tableofcontents

\section{Introduction}\label{Introduction}%%%%%%%%%%%%%%%%%%%%%%%%%%%%%%%%%%%%%%%%%%%

 Let $E/\mathbb Q$ be an elliptic curve with conductor $N_E$, and for each prime $p$, let $N_p=\#E_{\textrm{ns}}(\mathbb F_p)$, where $E_{\textrm{ns}}(\mathbb F_p)$ denotes the set of non-singular $\mathbb F_p$-rational points on a minimal Weierstrass model for $E$ at $p$.  We denote by $\rk(E)$ the rank of the Mordell-Weil group $E(\mathbb Q)$. The original version of the Birch and Swinnerton-Dyer conjecture takes the following form. 

\begin{conjecture} [Birch and Swinnerton-Dyer ~\cite{BirchSwinnertonDyer1965}]  \label{OBSD}
We have that $$\prod_{p \leq x} \frac{N_p}{p}  \sim C (\log x)^{\rk (E) }$$ as $x \to \infty$ for some constant $C$ depending on $E$. 
\end{conjecture}

The modern formulation of the conjecture is expressed in terms of the $L$-function $L(E, s)$ associated to $E$.  If we set $a_p=p+1-N_p$ if $p \nmid N_E$ and $a_p=p-N_p$ if $p |N_E$, then the $L$-function of $E$ is defined for $\textrm{Re}(s)>3/2$ by 
$$
L(E, s)= \prod_{p |N_E} (1-a_p p^{-s})^{-1} \cdot   \prod_{p \nmid N_E} (1-a_p p^{-s}+p^{1-2s})^{-1}. 
$$
By the work of Wiles ~\cite{Wiles1995} and Breuil--Conrad--Diamond--Taylor ~\cite{BreuilConradDiamondTaylor2001},  $L(E, s)$ admits an analytic continuation to the entire complex plane and satisfies a functional equation which relates $L(E, s)$ to $L(E, 2-s)$. More precisely,  if we define 
$
\Lambda(E, s) = N_E^{s/2} (2 \pi)^{-s} \Gamma(s) L(E, s)
$
the functional equation asserts that 
\begin{equation} \label{func}
\Lambda(E, s) = w_E \Lambda(E, 2-s) 
\end{equation}
for some $w_E \in \{\pm 1\}$. 
%Since $\Gamma(s)$ has poles at $s=0, -1, -2, \ldots$ it follows from the functional equation that $L(E, s)$ has trivial zeros at $s=0, -1, -2, \ldots$ \hspace{2mm}.  
In particular, the critical strip of $L(E, s)$ is the region $\frac{1}{2}< \Re(s)< \frac{3}{2}$ and the critical line of $L(E, s)$ is the line $\Re(s)=1$. The Riemann Hypothesis for $L(E, s)$ asserts that all the non-trivial zeros of $L(E, s)$ lie on the line $\textrm{Re}(s)=1$.   We now state the modern formulation of the Birch and Swinnerton-Dyer conjecture. 

\begin{conjecture} \label{BSD}
Let $E/\mathbb Q$ be an elliptic curve. 
Then
$$
\ord_{s=1} L(E, s)= \rk (E). 
$$
\end{conjecture}

Goldfeld \cite{Goldfeld1982} showed that Conjecture \ref{OBSD} implies both Conjecture \ref{BSD} and the Riemann Hypothesis for $L(E, s)$; moreover, he also gave an explicit expression for the constant $C$ appearing in Conjecture \ref{OBSD}.  

\begin{theorem}[Goldfeld \cite{Goldfeld1982}]
Let $E/\mathbb Q$ be an elliptic curve. If 
$$
\prod_{p \leq x} \frac{N_p}{p}  \sim C (\log x)^{\rk (E) }
$$
as $x \to \infty$, then $L(E, s)$ satisfies the Riemann Hypothesis and $\ord \limits _{s=1} L(E, s)= \rk(E)$. Moreover, if we set $r:=\ord \limits _{s=1}L(E,s )$,  then 
$$
C= \frac{r!}{L^{(r)} (E, 1) } \cdot \sqrt 2 e^{r \gamma}, 
$$
where $\gamma$ is  Euler's constant and $L^{(r)}(E, s)$ is the $r$-th derivative of $L(E, s)$.  

\end{theorem}
It is natural to ask, assuming the Riemann Hypothesis for $L(E, s)$ if necessary, whether the converse is true \textit{i.e.} whether Conjecture \ref{BSD} implies Conjecture \ref{OBSD}.

\subsection{Main results of the paper}
In this paper, we treat this question in a more general context by studying the behaviour of partial Euler products of $L$-functions of elliptic curves in the right-half of their critical strip. Our first main result establishes the asymptotic behaviour of these partial Euler products. 
To state our result, we first introduce the following notation: we set $r=\ord \limits_{s=1} L(E, s)$,  and let $\alpha_p$ and $\beta_p$ denote the Frobenius eigenvalues of $E$ at $p$. We let $ U_s(x)= \sum \limits _{\substack{\sqrt x < p \leq x \\ p \nmid  N_E}} \frac{(\alpha_p^2+\beta_p^2) }{2p^{2s}} $ and $
 R_s(x) = \frac{1}{\log x} \sum \limits_{ \rho \neq 1 } \frac{x^{\rho-s}}{ \rho-s}+ \frac{1}{ \log x} \sum \limits_{\rho \neq 1} \int_{s}^{\infty} \frac{x^{\rho-z}}{(\rho-z)^2}dz$, where the sums in $R_s(x)$ are taken over all non-trivial zeros $\rho=\beta+i \gamma$ of $L(E, s)$ (excluding $\rho=1$) counted with multiplicity and are interpreted as $\lim \limits_{T \to \infty} \sum \limits  _{ |\gamma| \leq T}$. For $s \in \mathbb C$ with $\Re(s)>1$ and $x \geq 2$, we let $I_s(x)= \int_{s}^{\infty} \frac{x^{1-z}}{1-z} dz$,  where the integral is taken along the horizontal straight line starting at $s$.

\begin{lettertheorem}[Theorem \ref{mainthm}] \label{asymptotic}
Assume the Riemann Hypothesis holds for $L(E, s)$. Then for a complex number $s \in \mathbb C$ with $1< \Re(s)< \frac{3}{2}$,  we have that 
\begin{equation*}
\prod_ {\substack{p \leq x \\ p | N_E}}(1-a_pp^{-s})^{-1} \cdot  \prod_ {\substack{p \leq x \\ p \nmid N_E}} (1-a_pp^{-s}+p^{1-2s})^{-1}  = L(E, s)  \exp \left (-r I_s(x) - R_s(x) 
+  U_s(x)+ O \left ( \frac{\log x}{x^{1/6}} \right ) \right),  \\
\end{equation*}
where the implied constant in the big O term above depends only on the elliptic curve $E$ and is independent of $s$. 
\end{lettertheorem}

Our theorem is an instance of the general paradigm that even though Euler products of $L$-functions are generally valid only to the right of the critical strip,  there is a strong sense in which they should also continue to persist inside the critical strip. Indeed, 
an analogous asymptotic for partial Euler products of the Riemann zeta function in the critical strip had been obtained by Ramanujan \cite {Ramanujan1915} during his study of highly composite numbers (see Remark \ref{Ramanujan}), which was recently generalized by Kaneko \cite{Kaneko2022} to the case of  Dirichlet $L$-functions.  The key idea in the proof of Theorem \ref{asymptotic}, as well as in the results of Ramanujan and Kaneko, is to develop a suitable version of an explicit formula (in the sense of analytic number theory) for the relevant $L$-function, that allows one to make a clear connection to the partial Euler products in question. We refer to Harper's Bourbaki survey \cite{Harper2019} (especially Principle 1.3 therein and the subsequent discussion) as well as the expository article \cite{KanekoKoyamaKurokawa2022} by Kaneko--Koyama--Kurokawa  for a more detailed discussion of the study of Euler products of $L$-functions in the critical strip. 

We use Theorem \ref{asymptotic} as a key tool to study the relations between the original and modern formulations of the Birch and Swinnerton--Dyer conjecture.  To state our results, we first introduce the following definition. 

\begin{definition}
Let $S \subseteq \mathbb R_{\geq 2}$ be a measurable subset of the real numbers. The logarithmic measure of $S$ is defined to be 
$$
\mu^{\times}(S)= \int_{S} \frac{dt}{t}. 
$$
\end{definition}

\begin{lettertheorem}[Theorem \ref{appln}] \label{app1} 
Assume the Riemann Hypothesis for $L(E, s)$. Then  there exists a subset $S \subseteq \mathbb R_{\geq 2}$ of finite logarithmic measure such that 
$$
\prod_{p \leq x} \frac{N_p}{p}  \sim C (\log x)^r   \hspace{2mm}  \text{ as }   x \rightarrow \infty \text{ with } x \not \in S , 
$$
where $\displaystyle r=\ord_{s=1} L(E, s)$, $C= \frac{r!}{L^{(r)} (E, 1) } \cdot \sqrt 2 e^{r \gamma}$, $\gamma$ is Euler's constant and $L^{ (r)}(E, s)$ is the $r$-th derivative of $L(E, s)$. 
\end{lettertheorem}

We briefly explain the proof of Theorem \ref{app1}. The key idea is to let $s$ tend to $1$ in Theorem \ref{asymptotic} in a controlled manner dependent on $x$; more precisely, we set $s=1+\frac{1}{x}$  and let $x \to \infty$.  A brief calculation shows that the left hand side of Theorem \ref{asymptotic} is asymptotic to $\prod_{p \leq x} \frac{p}{N_p}$.  The main work in the proof is then estimating the contributions from the various terms on the right hand side.   The term $I_s(x)$ contributes the main term $(\log x)^{r}$ in Theorem \ref{app1} (we show in Corollary \ref{ramanujancor} that $I_s(x)$ can be expressed in terms of the logarithmic integral for real $s$), while the term $U_s(x)$ contributes the factor of $\sqrt 2$ appearing in the constant $C$. This is the reason why we chose to keep the $U_s(x)$ term in its exact shape, rather than absorbing it into an error term.  The delicate issue in our proof is to handle the contribution coming from the zeros of $L(E, s)$ in the term $R_s(x)$.  By another explicit formula argument, we reduce the problem to estimating the sum  $\psi_E(x)=\sum  _{\substack{p^k \leq x \\ p\nmid N_E}} (\alpha_p^k+\beta_p^k) \log p$, where as above $\alpha_p$ and $\beta_p$ are the Frobenius eigenvalues at $p$.  The Riemann Hypothesis for $L(E, s)$ is equivalent to $\psi_E(x)=O(x (\log x)^2)$.  However, by a factor of $\log x$, this estimate is not good enough for our proof to carry through. Thus, in Section \ref{sec3} of the paper, we use a method of Gallagher \cite{Gallagher1980} to obtain, conditional on the Riemann Hypothesis for $L(E, s)$,  the slightly refined estimate $\psi_E(x)= O(x (\log \log x)^2)$ outside a set of finite logarithmic measure; this is the essential reason why the set $S$ of finite logarithmic measure appears in the statement of Theorem \ref{app1}.  The fact that the estimate coming from the Riemann Hypothesis is inadequate for our purposes is consistent with the results of \cite{Conrad2005} and \cite{KuoMurty2005}; we briefly discuss this in Section \ref{comparison} below. 
As a corollary to Theorem \ref{app1}, we recover Goldfeld's result \cite{Goldfeld1982} that Conjecture \ref{OBSD} implies Conjecture \ref{BSD}. 

\begin{lettercorollary}[Corollary \ref{cor2}] \label{OBSDimpliesBSD}
Let $E/\mathbb Q$ be an elliptic curve. Suppose that $$\prod_{p \leq x} \frac{N_p}{p}  \sim C (\log x)^{\rk(E) }.$$ Then 
$ \displaystyle{
\ord_{s=1} L(E, s)= \rk (E) }. 
$

\end{lettercorollary}

Moreover, we also prove a result in the direction of the converse; the Riemann Hypothesis for $L(E, s)$ and  Conjecture \ref{BSD} implies Conjecture $\ref{OBSD}$ off a set of finite logarithmic measure.

\begin{lettercorollary}[Corollary \ref{cor1}] \label{corB}
Let $E/\mathbb Q$ be an elliptic curve. Assume the Riemann Hypothesis for $L(E, s)$ and that $\ord \limits _{s=1} L(E, s)= \rk (E)$. Then there exists a subset $S \subseteq \mathbb R_{\geq 2}$ of finite logarithmic measure such that  
$$
\prod_{p \leq x} \frac{N_p}{p}  \sim C (\log x)^{\rk(E)} \hspace{2mm}  \text{ as }   x \rightarrow \infty \text{ with } x \not \in S,  
$$
where $C= \frac{r!}{L^{(r)} (E, 1) } \cdot \sqrt 2 e^{r \gamma}$, $\gamma$ is Euler's constant and $L^{ (r)}(E, s)$ is the $r$-th derivative of $L(E, s)$ with $r=\ord \limits _{s=1}L(E, s)$. 

\end{lettercorollary}

\subsection{Comparison with other work} \label{comparison}
Besides the work of Goldfeld \cite{Goldfeld1982} mentioned above, there is also related work by Conrad \cite{Conrad2005} and Kuo--Murty \cite{KuoMurty2005} on studying Conjecture \ref{OBSD} and its consequences. 
The new feature in our work is to study the asymptotic behaviour of the partial Euler products of $L(E, s)$ in the right-half of the critical strip, as opposed to studying its asymptotics just at the point $s=1$, as in the statement of Conjecture \ref{OBSD}. Working in this wider region allows us to deduce not only Goldfeld's result as a corollary of our main theorems, but also, as we briefly explain below, a result in \cite{Conrad2005} and \cite{KuoMurty2005}.  In \cite{Conrad2005} and \cite{KuoMurty2005}, Conrad and Kuo--Murty independently show that $
\prod_{p \leq x} \frac{N_p}{p}  \sim C (\log x)^r, 
$  is equivalent to the fact that $\psi_E(x)=o(x \log x)$. The latter estimate, while deeper than what can at present be obtained from the Riemann Hypothesis, is indeed plausible (we refer to \cite[page 290]{Conrad2005} for more details). 
Indeed, as mentioned above, we prove in Section \ref{sec3} that under the assumption of the Riemann Hypothesis, this estimate holds outside a set of finite logarithmic measure . The proof of Theorem \ref{app1} works verbatim, by replacing the estimate in Section \ref{sec3} with $\psi_E(x)=o(x \log x)$, to obtain the backward implication of the result of \cite{Conrad2005} and \cite{KuoMurty2005}.  In particular, assuming the estimate $\psi_E(x)=o(x \log x)$ in lieu of the Riemann Hypothesis also allows us to dispense with the set $S$ appearing in Corollary \ref{corB} . 

We point out that an important technique in our work that is common to many of the works cited above is the use of explicit formulas. However, as a technical comment, we mention that the explicit formulas used in our work are related directly to partial sums of Frobenius eigenvalues and in particular, unlike the explicit formula used in \cite{Goldfeld1982},  we do not need to use an auxiliary test-function in our versions of the explicit formula. 

Finally, we mention that Conjecture \ref{OBSD} has given rise to a general conjecture about the behaviour of partial Euler products of $L$-functions at the central point by Kurokawa et al. (see for instance \cite{KimuraKoyamaKurokawa2014} or \cite{KanekoKoyamaKurokawa2022}); the ideas introduced in this paper have been used in a sequel paper of the author \cite{Sheth2024} to study the relation between this conjecture and the Generalised Riemann Hypothesis, where our results have also been applied towards problems concerning Chebyshev's bias in the recently introduced framework of Aoki--Koyama  \cite{AokiKoyama2023}. 

\subsection*{Notational conventions} We write $f=O(g)$ or $f \ll g$ if there exists a positive constant $c$ such that $|f(z)| \leq c |g(z)|$ for all $z$ in a specified range. The constant $c$ is allowed to depend on the elliptic curve $E$.  The rank of $E$ is always denoted by $\rk(E)$ and the letter $r$ is reserved for $\ord \limits _{s=1} L(E, s)$.

\subsection*{Acknowledgements}
I would like to thank Adam Harper for very helpful discussions in connection with this paper and for feedback on a previous draft of this manuscript. I would also like to thank Nuno Arala Santos for helpful discussions regarding the proof of Theorem \ref{refinement}. Finally, I am very grateful to the anonymous referees for very helpful comments and suggestions.

\section{Partial Euler products in the critical strip}
Let $E/\mathbb Q$ be an elliptic curve. 
Since the Euler factor at a prime $p$ is quadratic for each $p \nmid N_E$, it will be convenient for us to linearize this Euler factor: for each $p \nmid N_E$, we let $\alpha_p$ and $\beta_p$ denote the roots of the polynomial $x^2-a_px +p$ and so
each Euler factor can be written as 
$$
(1-a_p p^{-s}+p^{1-2s})= (1-\alpha_p p^{-s}) \cdot (1-\beta_p p^{-s}). 
$$
The polynomial  $x^2-a_px +p$ can be interpreted as the characteristic polynomial of the $p$-power Frobenius morphism on the Tate-module $T_\ell(E)$, where $\ell \neq p$ is a prime (and so $\alpha_p$ and $\beta_p$ can be viewed as the corresponding eigenvalues). 
By the Riemann Hypothesis for elliptic curves over finite fields (see for example \cite[Theorem 2.3.1]{Sil}), we have that 
\begin{equation} \label{alphabeta}
|\alpha_p|=|\beta_p|=\sqrt p. 
\end{equation}

If $p|N_E$, it is a standard fact that  $a_p \in \{-1, 0, 1\}$.

\begin{lemma} \label{bn}
For all $s \in \mathbb C$ with $\Re(s)>3/2$, we have that 
$$
- \frac{L'(E, s)}{L(E, s)}=\sum_{n=1}^{\infty} \frac{b_n}{n^s}, 
$$
where  
\[ b_n= \begin{cases} 
      (\alpha_p^k+\beta_p^k) \cdot \log p & \textrm{ if } n=p^k \textrm{ and } p \nmid N_E \\
      a_p^{k} \cdot \log p &  \textrm{ if } n=p^k \textrm{ and } p | N_E \\
      0 & \textrm{otherwise.}
   \end{cases}
\]
\end{lemma}

\begin{proof}
This follows by a direct computation, see for example ~\cite[Lemma 5.1.4]{Spicer2015}. 
\end{proof}

To prove Theorem \ref{asymptotic}, we will prove a suitable version of an explicit formula (Theorem \ref{explicit} below) that will allow us to make a clear connection to the partial Euler products of $L(E, s)$; our explicit formula is an analogue of the corresponding explicit formula for the Riemann zeta function given in \cite[Lemma 5.2]{Gor23} (whose proof is given in the Appendix of \textit{op.cit.}). As a first step, we establish the following bound for the logarithmic derivative of $L(E, s)$, whose statement and proof is inspired from \cite[Lemma 12.4]{MV07}.

\begin{proposition} \label{stirling}
Let $ \mathcal A$ denote the set of points $s \in \mathbb C$ such that $\Re(s) \leq -1$ and $|s+k| \geq \frac{1}{4}$ for all $k \in \mathbb Z$. Then 
$\displaystyle 
{ \frac{L'(E, s) }{L(E, s)} \ll \log ( |s|+1) }$ uniformly for $s \in \mathcal A$. 
\end{proposition}

\begin{proof}
By the functional equation \eqref{func} of $L(E, s)$, we have that 
$$
L(E, s)= w_E N_E^{1-s} (2 \pi)^{2s-2} \frac{\Gamma(2-s)}{\Gamma(s)} L(E, 2-s). 
$$
Applying the reflection formula for $\Gamma(s)$ and the identity $\Gamma(s+1)=s \Gamma(s)$, we obtain that 
$$
\frac{\Gamma(2-s)}{\Gamma(s)} = \frac{ \sin(\pi s)  \Gamma(2-s) \Gamma(1-s) }{\pi}=  \frac{ \sin(\pi s)  (1-s) \Gamma(1-s)^2}{\pi}. 
$$
Substituting this back into the above equation and logarithmically differentiating yields 
\begin{equation*}
\frac{L'(E, s)}{L(E, s)}= -\frac{L'(E, 2-s)}{L(E, 2-s)}-\log(N_E)+2 \log (2 \pi) -\frac{1}{(1-s)} +\pi \cot(\pi s)- 2 \frac{\Gamma'(1-s)}{\Gamma(1-s)}. 
\end{equation*}
We note that the first four terms are bounded for $s \in \mathcal A$. To handle the last two terms, we note that $
\cot(\pi s)= i+\frac{2 i}{e^{2 \pi i s}-1} \ll 1
$
since $s$ is bounded away from all the integers, and that $\frac{\Gamma'(1-s)}{\Gamma(1-s)} \ll \log(|s|+1)$ by \cite[Theorem C.1]{MV07}. This completes the proof. 
\end{proof}

\begin{theorem} \label{explicit}
Let $s \in \mathbb C \setminus{ \{1\}} $ be a complex number such that $L(E, s) \neq 0$. For $x \geq 4$ and $T \geq 2+ |\Im(s)|$, we define $R(x, T, s)$ by 
$$
\sum_{n \leq x} {}^{'} \frac{b_n}{n^s} = -r \cdot \frac{x^{1-s}}{1-s}-\frac{L'(E, s)}{L(E, s)}-\sum_{\substack{\rho \neq 1 \\ |\Im (\rho -s)| \leq T}}\frac{x^{\rho-s}}{\rho-s}+\sum_{k=0}^{\infty} \frac{x^{-k-s}}{k+s} + R(x, T, s), 
$$
where $r=\ord \limits _{s=1} L(E, s)$, the sum over $\rho$ is taken over all non-trivial zeros of $L(E, s)$ (excluding $\rho=1$), and where the prime on the summation indicates that the last term of the sum is weighted by half if $x$ is an integer.  Then 
$$
R(x, T, s) \ll (\sqrt x \log x)  x'^{-\Re(s)} \min \left \{ 1 , \frac{x}{T \langle x \rangle } \right \} + \frac{ \log^2(xT) }{T} \left(  2^{|\Re(s)| } x^{\frac{3}{2}-\Re(s)} + \frac{2^{-\Re(s) }}{\log x} \right), 
$$
where $x'$ denotes the prime power closest to $x$ that is not equal to $x$ and $\langle x \rangle :=|x-x'|$. 
\end{theorem}

\begin{proof}
%This is a standard argument to prove explicit formulas in analytic number theory.  
We closely follow the proof of \cite[Lemma 5.2]{Gor23}. Let $\sigma_0= \max  \{ 0, \frac{3}{2}-\Re(s) \}$+ $\frac{1}{\log x}$. By using the truncated Perron's formula (\cite[Corollary 5.3]{MV07}) and applying Lemma \ref{bn}, we have that 
\begin{equation} \label{perron}
\sum_{n \leq x} {}^{'} \frac{b_n}{n^s} = %\frac{1}{2 \pi i} \int_{c- i \infty}^{c+ i \infty} \left( \sum_{n=1}^{\infty} \frac{b_n}{n^{s+z}} \right) \frac{x^z}{z} dz
 \frac{1}{2 \pi i} \int_{\sigma_0- i T}^{\sigma_0+ i T}  - \frac{L'(E, s+z)}{L(E, s+z)} \frac{x^z}{z} dz + E_s, 
\end{equation}
where 
$$
E_s \ll \sum_{\substack{\frac{x}{2} < n < 2x \\ n \neq x}} \frac{ |b_n| }{n^{\Re(s)} } \min \left \{ 1, \frac{x}{T |x-n| } \right \} + \frac{x^ {\sigma_0}}{T} \sum_{n=1}^{\infty} \frac{ |b_n| }{n^{\Re(s)+\sigma_0}}. 
$$

In the sum over $(x/2, 2x)$, we consider separately the case $n=x'$ and $ n \neq x'$ and use the bound in Equation $\eqref{alphabeta}$ to obtain 
\begin{align*}
E_s 
& \ll (\sqrt x \log x ) x'^{-\Re(s)} \min \left \{ 1 , \frac{x}{T \langle x \rangle } \right \} + \frac{2^{|\Re(s)|} x^{\frac{3}{2}-\Re(s)} \log^2 x }{T} +  \frac{x^ {\sigma_0}}{T} \sum_{n=1}^{\infty} \frac{ |b_n| }{n^{\Re(s)+\sigma_0}}. 
\end{align*}

To estimate the last sum above,  a brief case-by-case analysis (splitting into cases $\sigma_0+\Re(s) > \frac{5}{2}$ and $\sigma_0+\Re(s) \in (\frac{3}{2}, \frac{5}{2}]$) using the estimates $\sum_{n=1}^{\infty} \frac{|b_n|}{n^t} \ll \frac{\zeta'}{\zeta}(t-1/2)$ for $t > \frac{3}{2}$ (\cite[Lemma 5.1.7]{Spicer2015}), $\frac{\zeta'}{\zeta}(t) \asymp (t-1)^{-1}$ for $t \in (1, 2]$, and  $\frac{\zeta'}{\zeta}(t) \asymp 2^{-t}$ for $t \geq 2$ shows that $E_s$ can be absorbed in  $R(x, T, s)$. By \cite[Lemma 2.2 (i)]{Qu2007} there are $T_1, T_2 \in [T, T+1]$ such that 
\begin{equation} \label{qu}
\frac{L'}{L}(E, \sigma+ i \Im(s)- iT_2), \frac{L'}{L}(E, \sigma+ i \Im(s)+iT_1) \ll \log^2 T
\end{equation}
uniformly for $-\frac{3}{2} \leq  \sigma \leq \frac{5}{2}$.  We can extend the range of integration in Equation \eqref{perron} from $|\Im(z)| \leq T$ to $-T_2 \leq \Im(z) \leq T_1$ since the error incurred is at most $ \ll \frac{x^{\sigma_0}}{T} \left( -\frac{L'}{L} \right)(E, \sigma_0+ \Re(s)) $ which can be absorbed into the bound for $E_s$. Let $K > -\Re(s)$ denote a positive half integer which will be taken to $\infty$ and let $\mathcal C$ denote the contour consisting of three line segments connecting $\sigma_0-iT_2, -K-\Re(s)-iT_2, -K-\Re(s)+iT_1, \sigma_0+iT_1$. 
To apply Cauchy's residue theorem,  we compute the residues of the integrand inside the contour.

\begin{itemize}
\item When $s+z=1$, $\displaystyle{\res_{z=1-s} \left( \frac{L'(E, s+z)}{L(E,s+z)} \right)=r}$,   so   $\displaystyle{\res_{z=1-s} \left( - \frac{L'(E, s+z)}{L(E,s+z)} \frac{x^z}{z} \right)=-r \cdot \frac{x^{1-s}}{1-s}}$. 

\item When $z=0$, $\displaystyle{\res_{z=0} \left( - \frac{L'(E, s+z)}{L(E,s+z)} \frac{x^z}{z} \right)=-\frac{L'(E, s)}{L(E, s)}}$. 
\item If $\rho \neq 1$ is a non-trivial zero of $L(E, s)$, $\displaystyle{\res_{z=\rho-s} \left( - \frac{L'(E, s+z)}{L(E,s+z)} \frac{x^z}{z} \right)=- r_{\rho} \cdot \frac{x^{\rho-s}}{\rho-s}}$, where $r_\rho$ is the order of the zero $\rho$. Thus, counting multiplicity for all the non-trivial zeros of $L(E,s)$, we obtain a total contribution of $\displaystyle{\sum_{\substack{\rho \neq 1 \\ -T_2 < \Im (\rho -s) < T_1}}\frac{x^{\rho-s}}{\rho-s}}$. 
\item Finally, the trivial zeros of $L(E, s)$ are at $s=0, -1, -2, \ldots$ (which are all simple).  Thus, since  $\displaystyle{\res_{z=-k-s} \left( - \frac{L'(E, s+z)}{L(E,s+z)} \frac{x^z}{z} \right)=\frac{x^{-k-s}}{k+s}}$ for all $k \geq 0$, we get a total contribution of $\displaystyle{\sum_{0 \leq k < K} \frac{x^{-k-s}}{k+s}}$. 
\end{itemize}

Thus, by Cauchy's residue theorem, we obtain
\begin{align} \label{perron2}
\frac{1}{2 \pi i} & \int_{\sigma_0- i T}^{\sigma_0+ i T}  - \frac{L'(E, s+z)}{L(E, s+z)} \frac{x^z}{z} dz  = -r \cdot \frac{x^{1-s}}{1-s}-\frac{L'(E, s)}{L(E, s)}  -\sum_{\substack{\rho \neq 1 \\ -T_2 < \Im (\rho -s) < T_1}}\frac{x^{\rho-s}}{\rho-s} \nonumber \\ & + \sum_{0 \leq k <K} \frac{x^{-k-s}}{k+s}+   \frac{1}{2 \pi i} \int_{\mathcal C} - \frac{L'(E, s+z)}{L(E, s+z)} \frac{x^z}{z} dz. 
\end{align}
By using the standard fact that number of zeros $\rho$ with $T \leq  \Im(\rho) \leq T+1$ is $ \ll \log T$ (see for instance Theorem \ref{zeros} below), we can shorten the sum over $-T_2< \Im(\rho-s) <T_1$ to one over $-T \leq \Im(\rho-s) \leq T$ since the incurred error is 
$$
\ll \sum_{\Im(\rho-s) \in (T, T_1) \cup (-T_2, -T)} \frac{x^{\frac{3}{2}-\Re(s)}}{|\rho-s|} \ll \frac{x^{3/2-\Re(s)} \log T}{T},
$$
which can be absorbed into the error term for $R(x, T, s)$. We now estimate the integral over $\mathcal C$. To bound the integral on the horizontal sides, we consider separately three ranges of $\Re(z) \in [-K-\Re(s), \sigma_0]$. The contribution of $\Re(z) \in \left [-\frac{3}{2}-\Re(s), \min\{\frac{5}{2}-\Re(s), \sigma_0 \} \right]$ can be bounded using Equation \eqref{qu} to obtain 
\begin{equation*} 
\frac{1}{2 \pi i} \int_{-\frac{3}{2}-\Re(s)+iT_1}^{ \min\{\frac{5}{2}-\Re(s), \sigma_0 \}+iT_1} - \frac{L'(E, s+z)}{L(E, s+z)} \frac{x^z}{z} dz \ll \frac{\log^2 T x^{\min\{\frac{5}{2}-\Re(s), \sigma_0 \}}}{T \log x},
\end{equation*}
and the same bound holds if $T_1$ is replaced by $-T_2$. A brief calculation shows that this error term can be absorbed into the error term for $R(x, T, s)$. We  only consider the contribution from $\Re(z) \in (\frac{5}{2}-\Re(s), \sigma_0]$ when this is a non-empty interval. In this case, we again use the estimates $\frac{L'}{L}(E, t) \ll \frac{\zeta'}{\zeta}(t-1/2)$ for $t > \frac{3}{2}$ and $\frac{\zeta'}{\zeta}(t) \ll 2^{-t}$ for $t\geq 2$ to obtain that 
\begin{equation*}
\frac{1}{2 \pi i} \int_{\frac{5}{2}-\Re(s)+iT_1}^{\sigma_0+iT_1} - \frac{L'(E, s+z)}{L(E, s+z)} \frac{x^z}{z} dz \ll  \frac{2^{-\Re(s)}}{2^{\sigma_0}} \frac{x^{\sigma_0}}{T \log x} \ll \frac{2^{-\Re(s)}}{T}
\end{equation*}
 and we note that this error term can be absorbed into the error term for $R(x, T, s)$.  The same bound holds if $T_1$ is replaced by $-T_2$.
To bound the contribution of $\Re(z) \in [-K -\Re(s), -\frac{3}{2}-\Re(s)]$ we use Proposition \ref{stirling} to obtain 
\begin{equation*}
\frac{1}{2 \pi i} \int_{-K -\Re(s)+iT_1}^{ -\frac{3}{2}-\Re(s)+iT_1} - \frac{L'(E, s+z)}{L(E, s+z)} \frac{x^z}{z} dz \ll \int_{-K}^{-\frac{3}{2}} \log (T+|a|) \frac{x^{a-\Re(s)}}{T} da \ll \frac{\log T}{T} \frac{x^{- \frac{3}{2}-\Re(s) }}{\log x},
\end{equation*}
and this error term is also acceptable. The same bound holds if $T_1$ is replaced by $-T_2$.
The integral over the vertical line can also be bounded using Proposition \ref{stirling} to obtain
\begin{equation*}
\frac{1}{2 \pi i} \int_{-K-\Re(s)-iT_2}^{ -K-\Re(s)+iT_1} - \frac{L'(E, s+z)}{L(E, s+z)} \frac{x^z}{z} dz \ll \frac{\log(KT)}{K+\Re(s)} x^{-K-\Re(s)} \int_{-T_2}^{T_1} dt \ll  \frac{T \log(KT) x^{-K-\Re(s)}}{K+\Re(s)}
\end{equation*}
and this bound tends to 0 as $K$ tends to infinity. Combining Equations \eqref{perron} and \eqref{perron2} and inserting the bounds above completes the proof of the theorem.  
\end{proof}

To prove our main theorem, we will first relate the partial Euler product
\begin{equation*} 
\prod_ {\substack{p \leq x \\ p | N_E}}(1-a_pp^{-s})^{-1} \cdot  \prod_ {\substack{p \leq x \\ p \nmid N_E}}(1-a_pp^{-s}+p^{1-2s})^{-1}= \prod_ {\substack{p \leq x \\ p | N_E}}(1-a_pp^{-s})^{-1} \cdot   \prod_{\substack{p \leq x \\ p \nmid N_E}} (1-\alpha_p p^{-s})^{-1} (1-\beta_p p^{-s})^{-1}
\end{equation*}
to the series $\sum \limits_{n \leq x} \frac{b_n}{n^s}$.  We begin by noting that 
\begin{align*}  \label{deuler}
\frac{d}{ds}     \left( \log \prod_ {\substack{p \leq x \\ p | N_E}}(1-a_pp^{-s})^{-1} \prod_ {\substack{p \leq x \\ p \nmid N_E}} (1-a_pp^{-s}+p^{1-2s})^{-1} \right) = -\left (\text{I}(x)+\text{II}(x)+\text{III}(x) \right ) ,  \\
\end{align*} 
where 
\begin{equation*}
\text{I}(x) =  \sum_{\substack{p \leq x \\ p | N_E}} \frac {\log p \cdot a_p}{p^s-a_p}, \hspace{2mm} 
\text{I\hspace{-.1mm}I}(x) =  \sum_{\substack{p \leq x \\ p \nmid N_E}} \frac {\log p \cdot \alpha_p}{p^s-\alpha_p}, \hspace{2mm} 
\text{I\hspace{-.1mm}I\hspace{-.1mm}I}(x) =  \sum_{\substack{p \leq x \\ p \nmid N_E}} \frac {\log p \cdot \beta_p}{p^s-\beta_p}.
\end{equation*}

We now fix $s \in \mathbb C$ such that $\textrm{Re}(s)>1$. We write 
\begin{equation*}
 \text{I}(x) = \sum_{\substack{p \leq x \\ p | N_E}}  \frac{\log p  \cdot a_p}{p^s}+\sum_{\substack{p \leq x \\ p | N_E}}  \frac{\log p \cdot a_p^2}{p^{2s}}+ \sum_{k \geq 3} \sum_{\substack{p \leq x \\ p | N_E}} \frac{\log p  \cdot a_p^{k} }{p^{ks}}
\end{equation*}

\begin{equation*}
    \text{II}(x) =  \sum_{\substack{p \leq x \\ p \nmid N_E}}  \frac{\log p  \cdot \alpha_p}{p^s}+\sum_{\substack{p \leq x \\ p \nmid N_E}}  \frac{\log p \cdot \alpha_p^2}{p^{2s}}+ \sum_{k \geq 3} \sum_{\substack{p \leq x \\ p \nmid  N_E}} \frac{\log p  \cdot \alpha_p^{k} }{p^{ks}}
\end{equation*}

\begin{equation*} 
\text{III}(x) = \sum_{\substack{p \leq x \\ p \nmid N_E}}  \frac{\log p  \cdot \beta_p}{p^s}+\sum_{\substack{p \leq x \\ p \nmid N_E}} \frac{\log p \cdot \beta_p^2}{p^{2s}}+ \sum_{k \geq 3} \sum_{\substack{p \leq x \\ p \nmid N_E}} \frac{\log p  \cdot \beta_p^{k} }{p^{ks}}. 
\end{equation*}

On the other hand, using the definition of $b_n$ as given in Lemma \ref{bn}, 
\begin{equation*}
\sum_{n \leq x} \frac{b_n}{n^s}= {\sum_{\substack{p^k \leq x \\ p | N_E}}} \frac{a_p^k \cdot \log p}{p^{ks}} +  {\sum_{\substack{p^k \leq x \\ p \nmid  N_E}}} \frac{ (\alpha_p^k+\beta_p^k) \cdot \log p}{p^{ks}}. 
\end{equation*}

We can write the first term as
\begin{align*}
 {\sum_{\substack{p^k \leq x \\ p | N_E}}} \frac{a_p^k \cdot \log p}{p^{ks}}= {\sum_{\substack{p \leq x \\ p |  N_E}}}\frac{ \log p \cdot a_p }{p^{s}}+  \sum_{k \geq 2} {\sum_{\substack{p\leq x^{1/k} \\ p | N_E}}}\frac{ \log p \cdot a_p^k }{p^{ks}}
\end{align*}
and the second term as 
\begin{align*}
 {\sum_{\substack{p^k \leq x \\ p \nmid  N_E}}} \frac{ \log p \cdot (\alpha_p^k+\beta_p^k) }{p^{ks}} &=  {\sum_{\substack{p \leq x \\ p \nmid  N_E}}}\frac{ \log p \cdot (\alpha_p+\beta_p) }{p^{s}}+  {\sum_{\substack{p \leq \sqrt x \\ p \nmid  N_E}}}\frac{ \log p \cdot (\alpha_p^2+\beta_p^2) }{p^{2s}}+  \sum_{k \geq 3} {\sum_{\substack{p \leq x^{1/k} \\ p \nmid  N_E}}}\frac{ \log p \cdot (\alpha_p^k+\beta_p^k) }{p^{ks}}. \\
% &=  {\sum_{\substack{p \leq x \\ p \nmid  N_E}}}\frac{ \log p \cdot (\alpha_p+\beta_p) }{p^{s}}+  {\sum_{\substack{p^2 \leq x \\ p \nmid  N_E}}}\frac{ \log p \cdot (\alpha_p^2+\beta_p^2) }{p^{2s}}+   O \left (\sum_{p \leq x} \frac{\log p}{p^{3(s-1/2)}} \right), 
\end{align*}
\begin{comment}
where we used Equation \eqref{est} and the fact that 
$$ \sum_{k \geq 3} {\sum_{\substack{p^k \leq x \\ p \nmid  N_E}}}\frac{ \log p \cdot (\alpha_p^k+\beta_p^k) }{p^{ks}}= O \left (\sum_{k \geq 3} \sum_{p \leq x} \frac{\log p  \cdot \alpha_p^k}{p^{ks}} \right) .$$  
\end{comment}
Combining these observations it follows that 
\begin{align*}
\text{I}(x)+\text{II}(x)+\text{III}(x)  &= \sum_{n \leq x} \frac{b_n}{n^s}+ \sum_{\substack{\sqrt x < p \leq x \\ p \nmid  N_E}} \frac{ \log p \cdot (\alpha_p^2+\beta_p^2) }{p^{2s}}+ \sum_{k \geq 3} {\sum_{\substack{ x^{1/k}<p \leq x \\ p \nmid  N_E}}}\frac{ \log p \cdot (\alpha_p^k+\beta_p^k) }{p^{ks}} \\
&+  \sum_{k \geq 2} {\sum_{\substack{x^{1/k}<p \leq x \\ p | N_E}}}\frac{ \log p \cdot a_p^k }{p^{ks}}. 
\end{align*}
Thus, in summary, 
\begin{align}  \label{derivative}
\frac{d}{ds}     \left( \log \prod_ {\substack{p \leq x \\ p | N_E}}(1-a_pp^{-s})^{-1} \prod_ {\substack{p \leq x \\ p \nmid N_E}} (1-a_pp^{-s}+p^{1-2s})^{-1} \right) \newline
&=-\ \sum_{n \leq x} \frac{b_n}{n^s}- \sum _{\substack{\sqrt x < p \leq x \\ p \nmid  N_E}}\frac{ \log p \cdot (\alpha_p^2+\beta_p^2) }{p^{2s}} \nonumber \\
 -\sum_{k \geq 3} {\sum_{\substack{x^{1/k} <p  \leq x \\ p \nmid  N_E}}}\frac{ \log p \cdot (\alpha_p^k+\beta_p^k) }{p^{ks}}-  \sum_{k \geq 2} {\sum_{\substack{x^{1/k}<p \leq x \\ p | N_E}}}\frac{ \log p \cdot a_p^k }{p^{ks}}. 
\end{align} 
%and using Theorem \ref{explicit}, we have that 
%\begin{align}  \label{maineqn}
%\frac{d}{ds}     \left( \log \prod_ {\substack{p \leq x \\ p | N_E}}(1-a_pp^{-s})^{-1} \prod_ %{\substack{p \leq x \\ p \nmid N_E}} (1-a_pp^{-s}+p^{1-2s})^{-1} \right) 
%&=r \cdot \frac{x^{1-s}}{1-s}+\frac{L'(E, s)}{L(E, s)}+\sum_{\rho \neq 1} \frac{x^{\rho-s}}{\rho-s}-\sum_{k=0}^{\infty} \frac{x^{-k-s}}{k+s } \nonumber \\- \sum _{\substack{\sqrt x < p \leq x \\ p \nmid  N_E}}\frac{ \log p \cdot (\alpha_p^2+\beta_p^2) }{p^{2s}} 
%& -\sum_{k \geq 3} {\sum_{\substack{ x^{1/k}<p \leq x \\ p \nmid  N_E}}}\frac{ \log p \cdot (\alpha_p^k+\beta_p^k) }{p^{ks}}-  \sum_{k \geq 2} {\sum_{\substack{x^{1/k}<p \leq x \\ p | N_E}}}\frac{ \log p \cdot a_p^k }{p^{ks}}.
%\end{align} 

\begin{theorem} \label{mainthm} 
Assume the Riemann Hypothesis holds for $L(E, s)$. Then for a complex number $s \in \mathbb C$ with $1< \Re(s)< \frac{3}{2}$,  we have that 
\begin{equation*}
\prod_ {\substack{p \leq x \\ p | N_E}}(1-a_pp^{-s})^{-1} \cdot  \prod_ {\substack{p \leq x \\ p \nmid N_E}} (1-a_pp^{-s}+p^{1-2s})^{-1}  = L(E, s)  \exp \left (-r I_s(x) - R_s(x) 
+  U_s(x)+ O \left ( \frac{\log x}{x^{1/6}} \right ) \right),  \\
\end{equation*}
where $r=\ord \limits _{s=1} L(E, s)$,  $\displaystyle I_s(x)= \int_{s}^{\infty} \frac{x^{1-z}}{1-z} dz$,  $
\displaystyle R_s(x) = \frac{1}{\log x} \sum \limits_{ \rho \neq 1 } \frac{x^{\rho-s}}{ \rho-s}+ \frac{1}{ \log x} \sum \limits_{\rho \neq 1} \int_{s}^{\infty} \frac{x^{\rho-z}}{(\rho-z)^2}dz$ and  $ \displaystyle U_s(x)= \sum_{\substack{\sqrt x < p \leq x \\ p \nmid  N_E}} \frac{(\alpha_p^2+\beta_p^2) }{2p^{2s}} $. Here, the integral is taken along the horizontal straight line starting at $s$ and the sums are taken over all non-trivial zeros $\rho=1+i \gamma$ of $L(E, s)$ (excluding $\rho=1$) counted with multiplicity and are interpreted as $\lim \limits _{T \to \infty} \sum \limits _{ |\gamma| \leq T}$. 
\end{theorem}

\begin{proof}
For all $s \in \mathbb C$ with $1< \Re(s)< \frac{3}{2}$,  we have by Equation $\eqref{derivative}$ and Theorem $\ref{explicit}$ that 
\begin{align*}
\frac{d}{ds}   & \left( \log \prod_ {\substack{p \leq x \\ p | N_E}}(1-a_pp^{-s})^{-1} \prod_ {\substack{p \leq x \\ p \nmid N_E}}  (1-a_pp^{-s}+p^{1-2s})^{-1} \right) 
= r \cdot \frac{x^{1-s}}{1-s}+\frac{L'(E, s)}{L(E, s)}+ \sum_{\substack{\rho \neq 1 \\ |\Im (\rho -s)| \leq T}}\frac{x^{\rho-s}}{\rho-s}  \\
& - \sum_{k=0}^{\infty} \frac{x^{-k-s}}{k+s } 
- R(x, T, s)- \sum _{\substack{\sqrt x < p \leq x \\ p \nmid  N_E}}\frac{ \log p \cdot (\alpha_p^2+\beta_p^2) }{p^{2s}} 
 -\sum_{k \geq 3} {\sum_{\substack{ x^{1/k}<p \leq x \\ p \nmid  N_E}}}\frac{ \log p \cdot (\alpha_p^k+\beta_p^k) }{p^{ks}} \\ & -  
 \sum_{k \geq 2} {\sum_{\substack{x^{1/k}<p \leq x \\ p | N_E}}}\frac{ \log p \cdot a_p^k }{p^{ks}} - \frac{1}{2} \frac{b_x}{x^s}, 
\end{align*} 
where in the equation above, and in what follows subsequently,  the last term is only included if $x \in \mathbb N$. We now fix $s_0 \in \mathbb C$ with $1< \Re(s_0)< \frac{3}{2}$. Integrating this equation along the horizontal straight line from $s_0$ to $\infty$, and using the fact that the sum over the zeros is a finite sum, it follows that

\begin{align*} 
& \log \prod_ {\substack{p \leq x \\ p | N_E}}(1-a_pp^{-s_0})^{-1} \prod_ {\substack{p \leq x \\ p \nmid N_E}}  (1-a_pp^{-s_0}+p^{1-2s_0})^{-1}  =  -r\cdot \int_{s_0}^{\infty}\frac{x^{1-s}}{1-s} ds-\int_{s_0}^{\infty}\frac{L'(E, s)}{L(E, s)}ds 
 \nonumber \\
&-\sum_{\substack{\rho \neq 1 \\ |\Im (\rho -s)| \leq T}}\int_{s_0}^{\infty} \frac{x^{\rho-s}}{\rho-s}ds    + \sum_{k=0}^{\infty} 
\int_{s_0}^{\infty}\frac{x^{-k-s}}{k+s}ds + \int_{s_0}^{\infty} R(x, T, s)+
 \sum_{\substack{\sqrt x < p \leq x \\ p \nmid  N_E}} \int_{s_0}^{\infty} \frac{\log p \cdot (\alpha_p^2+\beta_p^2)}{p^{2s}}ds \nonumber \\ 
&+ \sum_{k \geq 3} \sum_{\substack{x^{1/k} < p \leq x \\ p \nmid  N_E}} \int_{s_0}^{\infty} \frac{\log p \cdot (\alpha_p^k+\beta_p^k)}{p^{ks}}ds 
+ \sum_{k \geq 2} \sum_{\substack{x^{1/k} < p \leq x \\ p | N_E}} \int_{s_0}^{\infty} \frac{\log p \cdot a_p^k }{p^{ks}}ds + \frac{1}{2} \int_{s_0}^{\infty} \frac{b_x}{x^s} ds . 
\end{align*} 

Using the bound for $R(x, T, s)$ in Theorem \ref{explicit} it follows that
\begin{equation*}
\int_{s_0}^{\infty} R(x, T, s) \ll \sqrt x \cdot x'^{-\Re(s_0)} \min \left \{ 1 , \frac{x}{T \langle x \rangle } \right \} + \frac{ \log^2(xT) }{T \log x} \left(  2^{|\Re(s_0)| } x^{\frac{3}{2}-\Re(s_0)} +  2^{-\Re(s_0) }\right)
\end{equation*}

and so we can let $T \to \infty$ to obtain
\begin{align} \label{logeulerproduct} 
& \log \prod_ {\substack{p \leq x \\ p | N_E}}(1-a_pp^{-s_0})^{-1} \prod_ {\substack{p \leq x \\ p \nmid N_E}}  (1-a_pp^{-s_0}+p^{1-2s_0})^{-1}  =  -r\cdot \int_{s_0}^{\infty}\frac{x^{1-s}}{1-s} ds-\int_{s_0}^{\infty}\frac{L'(E, s)}{L(E, s)}ds 
 \nonumber \\
&-\sum_{\rho \neq 1} \int_{s_0}^{\infty} \frac{x^{\rho-s}}{\rho-s}ds   + \sum_{k=0}^{\infty}
\int_{s_0}^{\infty}\frac{x^{-k-s}}{k+s}ds +
 \sum_{\substack{\sqrt x < p \leq x \\ p \nmid  N_E}} \int_{s_0}^{\infty} \frac{\log p \cdot (\alpha_p^2+\beta_p^2)}{p^{2s}}ds \nonumber \\ 
&+ \sum_{k \geq 3} \sum_{\substack{x^{1/k} < p \leq x \\ p \nmid  N_E}} \int_{s_0}^{\infty} \frac{\log p \cdot (\alpha_p^k+\beta_p^k)}{p^{ks}}ds 
+ \sum_{k \geq 2} \sum_{\substack{x^{1/k} < p \leq x \\ p | N_E}} \int_{s_0}^{\infty} \frac{\log p \cdot a_p^k }{p^{ks}}ds+ \frac{1}{2} \int_{s_0}^{\infty} \frac{b_x}{x^s} ds. 
\end{align} 

We now analyze each of these integrals in turn.

\begin{itemize}

\item The first integral equals $I_{s_0}(x)$ by definition.

\item We have that 
\begin{align*}
\int_{s_0}^{\infty} \frac{L'(E, s)}{L(E, s)} ds= -\log L(E, s_0). 
\end{align*}

\item We have that 
\begin{align*}
\sum_{\rho \neq 1} \int_{s_0}^{\infty} \frac{x^{\rho-s}}{\rho-s}ds &= \sum_{\rho \neq 1} \left(  \frac{x^{\rho-s}}{- \log x (\rho-s) } \Biggr|_{s_0}^{\infty}+ \frac{1}{\log x} \int_{s_0}^{\infty} \frac{x^{\rho-s}}{(\rho-s)^2}ds \right )  \\
&= \frac{1}{\log x} \sum_{\rho \neq 1} \frac{x^{\rho-s_0}}{\rho-s_0}+ \frac{1}{\log x} \sum_{\rho \neq 1} \int_{s_0}^{\infty} \frac{x^{\rho-s}}{(\rho-s)^2}ds  \\
%&= \frac{R_{s_0}(x)}{\log x} + \sum_{\rho} \int_{s_0}^{\infty} \frac{x^{\rho-s}}{(\rho-s)^2}ds.
&= R_{s_0}(x) \textrm{ by definition. }
\end{align*}

\item We note that 
\begin{equation*}
\sum_{k=0}^{\infty} \frac{x^{-k-s}}{k+s}
\ll \frac{1}{x^{\Re(s)}} \sum_{k=0}^{\infty}  \frac{x^{-k}}{k+1} \ll \frac{1}{x^{\Re(s)}}. 
\end{equation*}
Thus, 
\begin{equation*}
\int_{s_0}^{\infty} \sum_{k=0}^{\infty} \frac{x^{-k-s}}{k+s} ds \ll \int_{s_0}^{\infty} \frac{d|s|}{x^{\Re(s)} }= \frac{1}{x^{\Re(s_0)} \log x}. 
\end{equation*}

\item We have that 
\begin{equation*}
\sum_{\substack{\sqrt x < p \leq x \\ p \nmid  N_E}} \int_{s_0}^{\infty} \frac{\log p \cdot (\alpha_p^2+\beta_p^2)}{p^{2s}}ds = \sum_{\substack{\sqrt x < p \leq x \\ p \nmid  N_E}} \frac{(\alpha_p^2+\beta_p^2) }{2p^{2s_0}}=U_{s_0}(x) \textrm{ by definition. }
\end{equation*}

\item  We have that 
$$
 \sum_{k \geq 3} \sum_{\substack{x^{1/k} < p \leq x \\ p \nmid  N_E}} \int_{s_0}^{\infty} \frac{\log p \cdot (\alpha_p^k+\beta_p^k)}{p^{ks}}ds =  \sum_{k \geq 3} \sum_{\substack{x^{1/k} < p \leq x \\ p \nmid  N_E}} \frac{(\alpha_p^k+\beta_p^k) }{kp^{ks_0}} 
$$
Using Equation \eqref{alphabeta}, we have that 
\begin{equation*}
\sum_{x^{1/k}<p \leq x} \frac{\alpha_p^k+\beta_p^k}{k p^k} \ll \sum_{x^{1/k}<n} \frac{1}{k n^{k/2} } \ll \int_{x^{1/k} }^{\infty} \frac{dt}{k t^{k/2} }= \frac{x^{{1/k}-1/2}}{k^2/2-k} \ll \frac{x^{1/k}}{\sqrt x \cdot k^2}. 
\end{equation*}
Now, since we have that
\begin{equation*}
\sum_{k \geq 3} \frac{x^{1/k} }{k^2} = \sum_{3 \leq k \leq \log x} \frac{x^{1/k} }{k^2} +O(1) \ll x^{1/3} \log x, 
\end{equation*}
it follows that 
$$
\sum_{k \geq 3} \sum_{\substack{ x^{1/k} < p \leq x \\ p \nmid  N_E}} \frac{(\alpha_p^k+\beta_p^k) }{kp^{ks_0}} \ll \frac{\log x}{x^{1/6}}. 
$$
\end{itemize}

Finally, note that the penultimate term in Equation \eqref{logeulerproduct} is an empty sum when $x$ is large enough, and the last term satisfies 
$$
\frac{1}{2} \int_{s_0}^{\infty} \frac{b_x}{x^s} ds \ll \frac{b_x}{x^{ \Re(s_0)} \log x} \ll  \frac{\sqrt x}{x^{ \Re(s_0) }} \ll \frac{1}{\sqrt x}
$$
so this bound can be subsumed into the previous error term. Inserting all of the above estimates into \eqref{logeulerproduct} and exponentiating proves the result. 
\end{proof}

We recall that for all $x >0 $, $\Li(x)$ is defined to be the Cauchy principal value of $\int_{0}^x \frac{dt}{\log t}$; 
for the applications in Section \ref{Section4} of the paper, we record the following corollary.

\begin{corollary} \label{ramanujancor}
Assume the Riemann Hypothesis holds for $L(E, s)$. Then for  $s \in \mathbb R$ with $1< s < \frac{3}{2}$  we have, with the notation as above, that 
\begin{equation*}
\prod_ {\substack{p \leq x \\ p | N_E}}(1-a_pp^{-s})^{-1} \cdot  \prod_ {\substack{p \leq x \\ p \nmid N_E}} (1-a_pp^{-s}+p^{1-2s})^{-1}  = L(E, s)  \exp \left (-r \cdot \Li(x^{1-s}) - R_s(x) 
+  U_s(x)+ O \left ( \frac{\log x}{x^{1/6}} \right ) \right).   \\
\end{equation*}
\end{corollary}

\begin{proof}
We only need to show that $I_s(x)=\Li(x^{1-s})$ for $s >1$. By making the substitution $v=(1-z) \log x$ and $v=\log t$, we conclude as desired that 
\begin{equation*}
I_s(x) = \int_{s}^{\infty} \frac{x^{1-z}}{1-z} dz = \int_{-\infty}^{ (1-s) \log x} \frac{e^v}{v} dv = \int_{0}^{e^{(1-s) \log x}} \frac{dt}{\log t} = \textrm{Li}(x^{1-s}).  \qedhere
\end{equation*}

\end{proof}

\begin{remark} \label{Ramanujan}
During his study of highly composite numbers,  while investigating the maximal order of the divisor function, Ramanujan established the following formula for partial Euler products of the Riemann zeta function in the right-half of the critical strip:  for all $s \in \mathbb R$ with $\frac{1}{2} < s < 1$ we have under the assumption of the Riemann Hypothesis that 
$$
\prod_{p \leq x} (1-p^{-s})^{-1} = -\zeta(s) \exp \left (  \Li (\vartheta(x)^{1-s}) +  \frac{ 2s x^{\frac{1}{2}-s}}{(2s-1) \log x }+\frac{S_s(x)}{\log x} + O \left ( \frac{x^{\frac{1}{2}-s}}{ \log (x)^2} \right) \right), 
$$
where $\vartheta(x)=\sum \limits _{p \leq x} \log p$ is Chebyshev's function and 
$
S_s(x) = - s \sum \limits_{ \rho } \frac{x^{\rho-s}}{ \rho(\rho-s)}$, the sum taken over all non-trivial zeros of $\zeta(s)$. Ramanujan's initial investigations  were published in ~\cite{Ramanujan1915}, but his complete work on the subject, including the statement and proof of the above formula, remained in his ``lost notebook",  and was only published in 1997 in ~\cite{Ramanujan1997}. Corollary \ref{ramanujancor} can be regarded as an analogous formula for partial Euler products of $L$-functions of elliptic curves. 
\end{remark}

\section{Refinement of the error term for \texorpdfstring{$\psi_E(x)$}{psiE(x)} } \label{sec3}

Let $\psi_E(x)=\sum \limits_{\substack{p^k \leq x \\ p\nmid N_E}} (\alpha_p^k+\beta_p^k) \log p$. The Riemann Hypothesis for $L(E, s)$ is equivalent to the fact that 
$\psi_E(x)=O(x (\log x)^2)$. In this section, conditional on the Riemann Hypothesis for $L(E, s)$,  we use a method due to Gallagher \cite{Gallagher1980} to improve the error term outside a set of finite logarithmic measure.  To do so, we will need a general result (see for instance \cite{Selberg 1992}) on the vertical distribution of zeros for a wide class of $L$-functions.  For our elliptic curve $L$-function, the result is the following (see also \cite [Theorem 2]{MurtyKim2023}).

\begin{theorem} \label{zeros}
Let $N_E(t)$ be the number of zeros  $\rho=\beta+ i \gamma$ of $L(E, s)$ satisfying $0<\gamma<t$.  Then 
$$
N_E(t) = \frac{\alpha_E}{\pi} t(\log t+c) + O(\log t), 
$$
where $c$ is a constant and $\alpha_E$ is a constant depending on $E$. 
\end{theorem}

\begin{corollary} \label{convergencefact}
We have that
$$
\sum_{\rho} \frac{1} { |\rho|^2 } 
$$
converges, where the sum is taken over all non-trivial zeros of $L(E, s)$. 
\end{corollary}

\begin{proof}
Note that 
$$
\sum_{\rho} \frac{1}{|\rho|^2} = \sum_{t=0}^{\infty} \sum_{ t < |\gamma| \leq t+1} \frac{1}{|\rho|^2}. 
$$
By Theorem \ref{zeros}, the number of terms $t<\gamma \leq t+1$ is $\ll \log t$.  Thus, 
$$
\sum_{\rho} \frac{1}{|\rho|^2} \ll \sum_{t=1}^{\infty} \frac{\log t}{t^2}<\infty 
$$ \qedhere 
\end{proof}

We will also need the following estimate on mean values of exponential sums. 

\begin{lemma} \label{gallagher}
Let $\mathcal A$ be a discrete subset of $\mathbb R$.  For each $\nu \in \mathcal A$, let $c(\nu)$ be a complex number such that the series 
$$
S(u)= \sum_{\nu \in \mathcal A} c(\nu) e^{2 \pi i \nu u} 
$$
is absolutely convergent. Then for any $\theta \in (0, 1)$, we have that
$$
\int_{-U}^{U} |S(u)|^2 du \leq  \left( \frac{\pi \theta}{\sin (\pi \theta)} \right)^{2} \int_{-\infty}^{\infty} \left | \frac{U}{\theta} \sum_{t \leq \nu \leq t+ \frac{\theta}{U}} c(\nu) \right |^2 dt .
$$
\end{lemma}

\begin{proof}
This lemma was originally proven by Gallagher (see \cite[Lemma 1]{Gallagher1970}), but the version stated above is \cite[Lemma 1]{Koyama2016}. 
\end{proof}

\begin{theorem} \label{refinement}
Assume the Riemann Hypothesis for $L(E, s)$.  Then there exists a subset $S \subseteq \mathbb R_{\geq 2}$ of finite logarithmic measure such that
$$
\psi_E(x)= O(x (\log \log x)^2) \hspace{2mm}  \text{ as }   x \rightarrow \infty \text{ with } x \not \in S.  
$$
\end{theorem}

\begin{proof}
Let $X$ be such that $X \leq x \ll X$. By the Riemann von-Mangoldt explicit formula for $L(E, s)$ (see \cite[Lemma 2.1]{Fiorilli2014} and \cite[Equation 1.10]{Rubenstein 2013}), we have that
\begin{equation*}
\psi_E(x) = - \sum_{\substack{\rho \\ |\gamma|  \leq X}} \frac{x^{\rho}}{\rho}+O(\log^2 X). 
\end{equation*}

By Theorem \ref{zeros}, the number of terms $t<\gamma \leq t+1$ is $\ll \log t$.  Thus, 
\begin{equation*}
\sum_{ t< |\gamma| \leq  t+1} \frac{x^\rho}{\rho}  \ll x \cdot \frac{\log t}{t} 
\end{equation*}
and hence

\begin{align} \label{gallagher12} 
\sum_{0< |\gamma|< \log X} \frac{x^\rho}{\rho} \ll X \sum_{j=1}^{ \lfloor \log X \rfloor  } \frac{\log(j+1)}{j}  \ll X \int_{1}^{\log X} \frac{\log t}{t} dt= \frac{X (\log \log X)^2}{2} \ll x( \log  \log x)^2. 
\end{align}

Let $T \leq X$. We now consider the integral 
$$
\int_X^{eX} \left| \sum_{T<\gamma \leq X} \frac{x^\rho}{\rho} \right|^2 \frac{dx}{x^3}. 
$$
By making the substitution $x=e^u$ and then $u=\frac{1}{2}+\log X+ 2\pi t$, this integral equals
$$
\int_{\log X}^{1+\log X} \left| \sum_{T<\gamma<X} \frac{e^{i \gamma u}}{\rho} \right|^2 du = 2\pi \cdot  \int_{-\frac{1}{4\pi}}^{\frac{1}{4 \pi} } \left| \sum_{T<\gamma<X} \frac{e^{i \gamma (\frac{1}{2}+\log X) }}{\rho} \cdot e^{2 \pi i \gamma t}  \right|^2 dt. 
$$
Setting $\mathcal A=\{ \gamma: T< \gamma \leq X\}$,  $\theta=\frac{1}{4 \pi }$,
and  \[c(\gamma)= 
\begin{cases}
\frac{e^{i \gamma (\frac{1}{2}+\log X)}}{\rho} & \text{if } T<\gamma \leq X \\
 0& \text{otherwise} 
\end{cases}
\]
we can apply Lemma \ref{gallagher} conclude that 
\begin{equation*}
\int_X^{eX} \left| \sum_{T<\gamma \leq X} \frac{x^\rho}{\rho} \right|^2 \frac{dx}{x^3} \ll \int_{-\infty}^{\infty}  \left( \sum_{\substack{ T <\gamma \leq X \\ t \leq \gamma \leq t+1}} \frac{1}{|\rho|}  \right)^2 dt \ll \int_{T-1}^{X}  \left( \sum_{t \leq \gamma \leq t+1} \frac{1}{|\rho|}  \right)^2 dt,  
\end{equation*}
Using again the fact that the number of terms $t<\gamma \leq t+1$ is $\ll \log t$ we get that 
\begin{align*}
 \int_{T-1}^{X}  \left( \sum_{t<\gamma<t+1} \frac{1}{|\rho|}  \right)^2 dt   & \ll  \int_{T-1}^{X}  \left( \frac{\log t}{t}  \right)^2 dt \\
 &= \frac{-\log ^2 t- 2 \log t-2}{t} \Biggr|_{T-1}^{X} \\
 & \ll \frac{\log ^2 T}{T}. 
\end{align*}
Define a set $S_T$ by setting 
$$
S_T = \{ x \in [e^T, e^{T+1}]: \left| \sum_{T < |\gamma| \leq e^T} \frac{x^\rho}{\rho} \right|  \geq x( \log \log x)^2 \}. 
$$
Setting $X=e^{T}$, we conclude from above that 
\begin{align*}
\frac{\log^2 T}{T} & \gg \int_{S_T} \left| \sum_{T<\gamma \leq X} \frac{x^\rho}{\rho} \right|^2 \frac{dx}{x^3} \\
 & \gg \int_{S_T} \frac{(\log \log x)^4}{x} dx \\
 & \gg  (\log \log e^T)^4 \int_{S_T} \frac{dx}{x} \\
 & =  (\log T)^4 \cdot \mu^{\times}(S_T). 
\end{align*}
Thus, 
$$
\mu^{\times}(S_T) \ll \frac{1}{T \log^2 T}. 
$$
Define the set 
$$
S= \bigcup_{T=2}^{\infty} S_T. 
$$
Then 
$$
\mu^{\times}(S) \ll \sum_{T=2}^{\infty} \frac{1}{T \log^2 T}< \infty. 
$$
Thus, $S$ has finite logarithmic measure and so, by definition of $S$ and by Equation \eqref{gallagher12}, we conclude that $
\psi_E(x)= O(x (\log \log x)^2 $  $\text{as }   x \rightarrow \infty \text{ with } x \not \in S$. 

\end{proof}

\section{Applications} \label{Section4}

In this section, we use our asymptotic for the partial Euler products of $L(E, s)$ in Theorem \ref{mainthm} and Corollary \ref{ramanujancor} to study the relations between the original and modern formulations of the Birch and Swinnerton-Dyer conjecture.  We begin with the following lemma. 

\begin{lemma} \label{sqrt2}
Let $U_s(x)$ be as in Theorem \ref{mainthm}. Then we have that 
$$
\lim_{x \to \infty} U_1(x) = \log \left( \frac{1}{\sqrt 2} \right). 
$$
\end{lemma}

\begin{proof}
This follows from the fact that (see \cite[page 332]{KuoMurty2005}) 
\begin{equation} \label{Mertens}
  \sum_{\substack{p \leq x \\ p \nmid  N_E}} \frac{(\alpha_p^2+\beta_p^2) }{2p^{2}}= -\frac{1}{2} \log \log x+ B +o(1), 
\end{equation}
where $B$ is a constant, since we then obtain
\begin{align*}
U_1(x) &= -\frac{1}{2} \log \log x+ \frac{1}{2} \log \log \sqrt x+ o(1) \\ &= \log \left( \frac{1}{\sqrt 2} \right)+o(1). 
\end{align*}

Equation \eqref{Mertens} is a generalized version of Mertens' theorem; we refer to \cite[Lemma 3.5]{Sheth2024} for a more general version of this statement statement from which it can also be deduced. 

\end{proof}

\begin{theorem} \label{appln}
Assume the Riemann Hypothesis for $L(E, s)$. Then  there exists a subset $S \subseteq \mathbb R_{\geq 2}$ of finite logarithmic measure such that 
$$
\prod_{p \leq x} \frac{N_p}{p}  \sim C (\log x)^r \hspace{2mm}  \text{ as }   x \rightarrow \infty \text{ with } x \not \in S,  
$$
where $r=\ord \limits _{s=1} L(E, s)$,  $C= \frac{r!}{L^{(r)} (E, 1) } \cdot \sqrt 2 e^{r \gamma}$, $\gamma$ is Euler's constant and $L^{ (r)}(E, s)$ is the $r$-th derivative of $L(E, s)$. 
\end{theorem}

\begin{proof}
The idea of the proof is to set $s=1+\frac{1}{x}$ and estimate the left-hand and right-hand sides of Corollary \ref{ramanujancor} as $x \to \infty$.  When $s=1+\frac{1}{x}$, we have that
\begin{align*}
\prod_ {\substack{p \leq x \\ p \nmid N_E}} (1-a_pp^{-s}+p^{1-2s})^{-1} &= \prod_ {\substack{p \leq x \\ p \nmid N_E}} (1-a_pp^{-1-1/x}+p^{-1-2/x})^{-1} \\
&= \prod_ {\substack{p \leq x \\ p \nmid N_E}} \left (1-a_pp^{-1} \left (1+ O \left ( \frac{\log p}{x} \right ) \right)+p^{-1} \left (1+ O \left ( \frac{\log p}{x} \right ) \right) \right)^{-1}  \\
&= \prod_ {\substack{p \leq x \\ p \nmid N_E}} (1-a_pp^{-1}+p^{-1})^{-1} \cdot \prod_ {\substack{p \leq x \\ p \nmid N_E}} \left(1 + O \left ( \frac{a_p \cdot \log p }{p  x}  \right) \right), 
\end{align*}
where we used the fact that 
$$
(1+a+O(f(x)))^{-1}= (1+a)^{-1} (1+O(f(x))) \textrm{ if } f(x)=o(1) \textrm{ and } a  \textrm{ is sufficiently small.  } 
$$
Now 
$$
 \prod_ {\substack{p \leq x \\ p \nmid N_E}} \left(1 + O \left ( \frac{a_p \cdot \log p }{p  x}  \right) \right) = 1+O \left (  \sum_ {\substack{p \leq x \\ p \nmid N_E}} \frac{|a_p| \cdot \log p}{p x} \right ) = 1+O \left ( \frac{\log x}{\sqrt x} \right) =1+o(1), 
$$
where we used the fact that $|a_p| \leq 2 \sqrt p$ and the classical estimate $\sum_{p \leq x} \frac{\log p}{p}=\log x +O(1)$. 
Thus, when $s=1+\frac{1}{x}$, 
\begin{equation*} 
\prod_ {\substack{p \leq x \\ p \nmid N_E}} (1-a_pp^{-s}+p^{1-2s})^{-1} \sim \prod_ {\substack{p \leq x \\ p \nmid N_E}} (1-a_pp^{-1}+p^{-1})^{-1} = \prod_ {\substack{p \leq x \\ p \nmid N_E}} \frac{p}{N_p}. 
\end{equation*}
An analogous argument shows that 
\begin{equation*} 
\prod_ {\substack{p \leq x \\ p | N_E}} (1-a_pp^{-s})^{-1} \sim \prod_ {\substack{p \leq x \\ p | N_E}} (1-a_pp^{-1})^{-1} = \prod_ {\substack{p \leq x  \\ p | N_E}} \frac{p}{N_p}. 
\end{equation*}
Combining the two previous equations yields that for $s=1+\frac{1}{x}$, 

\begin{equation} \label{lhs}
\prod_ {\substack{p \leq x \\ p | N_E}} (1-a_pp^{-s})^{-1} \prod_ {\substack{p \leq x \\ p \nmid N_E}} (1-a_pp^{-s}+p^{1-2s})^{-1} \sim \prod_ {p \leq x} \frac{p}{N_p}. 
\end{equation}

We now estimate the right-hand side of Corollary \ref{ramanujancor}.  Write 
$$
L(E, s)= a_r (s-1)^r+ a_{r+1} (s-1)^{r+1}+ \cdots 
$$
so that when $s=1+\frac{1}{x}$,  $L(E, s) \sim a_r \cdot \frac{1}{x^r}$ as $x \to \infty$. To estimate the contribution coming from the term $\Li(x^{1-s})$, we use the classical fact (see for instance \cite[pp.425]{Finch2003} or  \cite[Equation (2.2.5)]{Hardy1940}) that
$$
\Li(x) = \gamma + \log |\log x| +\sum_{n=1}^{\infty} \frac{ (\log x)^n}{n! \cdot n} \textrm{ for  all } x \in \mathbb  R_{>0} \setminus \{1\} .
$$
Applying this when $s=1+\frac{1}{x}$ yields that 
$$
\Li(x^{1-s}) = \gamma+ \log \left (\frac{\log x}{x} \right )+ o(1). 
$$

To estimate the contribution coming from the term $U_s(x)$ we note that $U_{1+\frac{1}{x}}(x) \sim U_1(x)$ since 
\begin{align*}
\sum_{{\substack{\sqrt x < p \leq x \\ p \nmid  N_E}}} \frac{(\alpha_p^2+\beta_p^2) }{2p^{2(1+\frac{1}{x}) }}&= \sum_{{\substack{\sqrt x < p \leq x \\ p \nmid  N_E}}} \frac{(\alpha_p^2+\beta_p^2) }{2p^{2}} \cdot \left(1+ O\left (\frac{\log p}{x} \right) \right) \\
&= \sum_{{\substack{\sqrt x < p \leq x \\ p \nmid  N_E}}} \frac{(\alpha_p^2+\beta_p^2) }{2p^{2}}+ O\left (\frac{1}{x} \sum_{\sqrt x < p \leq x} \frac{ \log p}{p} \right)= \sum_{{\substack{\sqrt x < p \leq x \\ p \nmid  N_E}}} \frac{(\alpha_p^2+\beta_p^2) }{2p^{2}}+ o(1). 
\end{align*}

Using Lemma \ref{sqrt2}, we conclude that when $s=1+\frac{1}{x}$, $U_s(x) \to \log \left( \frac{1}{\sqrt 2} \right)$ as $x \to \infty$. 

To estimate the contribution coming from the term $R_s(x)= \frac{1}{\log x}\sum \limits_{ \rho \neq 1} \frac{x^{\rho-s}}{ \rho-s}+\frac{1}{\log x} \sum \limits_{\rho \neq 1} \int_{s}^{\infty} \frac{x^{\rho-z}}{(\rho-z)^2}dz$  when $s=1+\frac{1}{x}$,  we begin by noting that 
\begin{equation*}
\sum_{\rho \neq 1} \frac{1}{\rho-s}-\sum_{\rho \neq 1} \frac{1}{\rho} = \sum_{\rho \neq 1} \frac{s}{ (\rho-s) \rho}
\end{equation*}
and using Corollary \ref{convergencefact},  we have that 
\begin{equation*}
\sum_{\rho \neq 1 } \frac{s}{ (\rho-s) \rho} \ll  \sum_{\rho} \frac{1}{|\rho|^2} < \infty. 
\end{equation*}
This implies that 
\begin{equation*}
\sum_{\rho \neq 1} \frac{1}{\rho-s} =\sum_{\rho} \frac{1}{\rho}+O(1). 
\end{equation*}

Since we are assuming the Riemann Hypothesis for $L(E, s)$, a similar argument to the one above yields
$$
\sum_{\rho \neq 1} \frac{x^{\rho-s}}{\rho-s}=\frac{1}{x^s} \sum_{\rho} \frac{x^\rho}{\rho}+O(x^{1-s}) \ll \frac{1}{x} \sum_{\rho} \frac{x^\rho}{\rho} +O(1). 
$$
Using the Riemann von-Mangoldt explicit formula for $L(E, s)$ (see  \cite[Equation 1.10]{Rubenstein 2013})
$$
\psi_E(x) =- \sum_{\rho} \frac{x^{\rho}}{\rho}+o(x^{1/2}), 
$$
it follows that 
$$
\sum_{\rho \neq 1} \frac{x^{\rho-s}}{\rho-s} \ll \frac{1}{x} \psi_E(x)+O(1). 
$$
By Theorem \ref{refinement}, we conclude that there exists a set $S$ of finite logarithmic measure such that $\text{as }   x \rightarrow \infty \text{ with } x \not \in S$, 
$$
\sum_{\rho \neq 1} \frac{x^{\rho-s}}{\rho-s} \ll \frac{1}{x} O( x (\log \log x)^2)+O(1) 
$$
and so $\text{as }   x \rightarrow \infty \text{ with } x \not \in S$, 
$$
\frac{1}{\log x} \sum_{\rho \neq 1} \frac{x^{\rho-s}}{\rho-s}  =o(1). 
$$
To estimate the second quantity in the definition of $R_s(x)$, we note that since $\sum \limits_{\rho} \frac{1}{|\rho|^2}$ converges by Corollary \ref{convergencefact}, 
\begin{align*}
\frac{1}{\log x} \sum_{\rho \neq 1} \int_{s}^{\infty} \frac{x^{\rho-z}}{(\rho-z)^2}dz &= O \left(\frac{1}{\log x} \cdot \int_ {s}^{\infty}  x^{1-\Re(z) } d|z| \right) \\
&= O \left ( \frac{x^{1 -\textrm{Re}(s)}}{\log^2 x}  \right) 
=o(1).  \\ 
\end{align*}
 Thus, in summary, when $s=1+\frac{1}{x}$, we conclude that $\text{as }   x \rightarrow \infty \text{ with } x \not \in S$, 
\begin{align*}
  & L(E, s)    \exp   \left (-r \cdot \Li(x^{1-s}) -R_s(x) 
+  U_s(x) +  O \left ( \frac{\log x}{x^{1/6}} \right) \right)  \\
& = L(E, s)  \exp \left ( -r\gamma - r \log \left (\frac{\log x}{x} \right )+ \log \left( \frac{1}{\sqrt 2} \right)+ o(1) \right ) \\
& \sim  \frac{a_r}{x^r} \exp \left (  -r\gamma -r \log \left (\frac{\log x}{x} \right)+ \log \left( \frac{1}{\sqrt 2} \right) \right) 
 = \frac{a_r}{\sqrt 2 \cdot e^{r \gamma} } \cdot \frac{1}{ (\log x)^r}. 
\end{align*}
Combining this with Equation \eqref{lhs} proves the theorem. 
\end{proof}

\begin{corollary} \label{cor2} 
Let $E/\mathbb Q$ be an elliptic curve. Suppose that $$\prod_{p \leq x} \frac{N_p}{p}  \sim C (\log x)^{\rk (E) }.$$ Then 
$$
\ord_{s=1} L(E, s)= \rk (E). 
$$

\end{corollary}

\begin{proof}
Assuming the Riemann Hypothesis for $L(E, s)$, this follows directly from Theorem \ref{appln}.  It is known that  (see \cite[Theorem 2(a)]{Goldfeld1982} or \cite[pp.334]{KuoMurty2005}) that the asymptotic $\prod_{p \leq x} \frac{N_p}{p}  \sim C (\log x)^{\rk (E(\mathbb Q))}$ implies the Riemann Hypothesis for $L(E, s)$, so the corollary is true unconditionally. For completeness, we explain the proof of this implication here. Note that 
$$
 \prod_{p \leq x} \frac{N_p}{p} = \prod_{p|N_E}\left(1-\frac{a_p}{p} \right)  \prod_{p \nmid N_E} \left(1-\frac{\alpha_p}{p} \right) \left(1-\frac{\beta_p}{p} \right). 
$$
Taking logarithms on both sides yields and computing the remainder term (see \cite[Lemma 1]{KuoMurty2005} for details) yields that 
$$
\log \prod_{p \leq x} \frac{N_p}{p}= \sum_{n \leq x}c_n+ O(1), 
$$
where \[ c_n= \begin{cases} 
      -\frac{(\alpha_p^k+\beta_p^k)}{kp^k} & \textrm{ if } n=p^k \textrm{ and } p \nmid N_E \\
     -\frac{a_p^k}{k p^k}  &  \textrm{ if } n=p^k \textrm{ and } p | N_E \\
      0 & \textrm{otherwise.}
   \end{cases}
\]
Thus, by hypothesis, $\sum \limits _{n \leq x} c_n =O( \rk (E) \cdot \log \log x)$ and partial summation yields that 
$$
\log L(E,s+1)= -\sum_{n=1}^{\infty} \frac{c_n}{n^s}= -s \int_{1}^{\infty} (\sum_{n \leq x}c_n)  x^{-(s+1)}dx \textrm{ for all } s \in \mathbb C \textrm{ with } \Re(s)>1/2. 
$$
Hence, $\log L(E, s+1)$ is holomorphic for all $s \in \mathbb C$ with $\Re(s)>0$ and so $L(E, s) \neq 0$ for all $s \in \mathbb C$ with $\Re(s)>1$. 
\end{proof}

\begin{corollary} \label{cor1}
Let $E/\mathbb Q$ be an elliptic curve. Assume that $\ord \limits _{s=1} L(E, s)= \rk(E)$ and the Riemann Hypothesis for $L(E, s)$. Then there exists a subset $S \subseteq \mathbb R_{\geq 2}$ of finite logarithmic measure such that
$$
\prod_{p \leq x} \frac{N_p}{p}  \sim C (\log x)^{\rk(E)} \hspace{2mm}  \text{ as }   x \rightarrow \infty \text{ with } x \not \in S, 
$$
where $C= \frac{r!}{L^{(r)} (E, 1) } \cdot \sqrt 2 e^{r \gamma}$, $\gamma$ is Euler's constant and $L^{ (r)}(E, s)$ is the $r$-th derivative of $L(E, s)$ with $r=\ord \limits _{s=1}L(E, s)$. 
\end{corollary}

\begin{proof}
This follows from Theorem  \ref{appln} once we have identified $\ord \limits_ {s=1} L(E, s)$ with $\rk(E)$. 
\end{proof}

\end{document}